\documentclass[a4paper]{amsart}
\usepackage{amsthm,amsmath,amssymb,mathrsfs}
\usepackage[latin1]{inputenc}
\usepackage{mathrsfs}
\usepackage{url}
\usepackage{mathtools}

\usepackage{tikz} 
\usetikzlibrary{matrix,arrows,calc,decorations.pathreplacing,fit}
\usetikzlibrary{fadings}
\usepackage{tikz-cd}
\usepackage{lua-visual-debug}

\numberwithin{equation}{section}

\newtheorem{theorem}[equation]{Theorem}
\newtheorem{lemma}[equation]{Lemma}

\newtheorem{corollary}[equation]{Corollary}
\newtheorem{cordef}[equation]{Corollary/Definition}
\newtheorem{proposition}[equation]{Proposition}
\newtheorem{propdef}[equation]{Proposition/Definition}

\newtheorem*{nothm}{Theorem}
\theoremstyle{definition}
\newtheorem{definition}[equation]{Definition}

\theoremstyle{remark}
\newtheorem{remark}[equation]{Remark}

\DeclareMathOperator{\diminj}{diminj}

\DeclareMathOperator{\id}{id}

\DeclareMathOperator{\Hom}{Hom}

\DeclareMathOperator{\Spec}{Spec}

\def \AA {\mathbb{A}}

\def \CC {\mathbb{C}}

\def \NN {\mathbb{N}}

\def \ZZ {\mathbb{Z}}

\def \Fcal {\mathcal{F}}

\def \mfr {\mathfrak{m}}

\def \pfr {\mathfrak{p}}
\def \qfr {\mathfrak{q}}

\def \Ascr {\mathscr{A}}
\def \Bscr {\mathscr{B}}

\def \Fscr {\mathscr{F}}

\def \Oscr {\mathscr{O}}

\def \Pbar {\bar{P}}

\def \Xbar {\bar{X}}

\def \fbar {\bar{f}}
\def \gbar {\bar{g}}
\def \hbar {\bar{h}}

\def \ybar {\bar{y}}

\def \Htilde {\tilde{H}}

\def \mono  {\hookrightarrow}

\def \isom  {\stackrel{\sim}{\rightarrow}}

\newcommand{\restr}[2]{{#1}\raise-.5ex\hbox{\ensuremath|}_{#2}}

\newcounter{subenvcounter}
\newenvironment{subenv}{%
 \begin{list}
  {\em (\arabic{subenvcounter})}
  {\setlength{\leftmargin}{20pt}
   \setlength{\rightmargin}{0pt}
   \setlength{\itemindent}{0pt}
   \setlength{\labelsep}{5pt}
   \setlength{\labelwidth}{13pt}
   \setlength{\listparindent}{\parindent}
   \setlength{\parsep}{0pt}
   \setlength{\itemsep}{0pt}
   \setlength{\topsep}{-\parskip}
   \usecounter{subenvcounter}}}
  {\end{list}}

\newcounter{asslistcounter}
\newenvironment{assertionlist}{
 \begin{list}
  {\upshape (\alph{asslistcounter})}
  {\setlength{\leftmargin}{18pt}
   \setlength{\rightmargin}{0pt}
   \setlength{\itemindent}{0pt}
   \setlength{\labelsep}{5pt}
   \setlength{\labelwidth}{13pt}
   \setlength{\listparindent}{\parindent}
   \setlength{\parsep}{0pt}
   \setlength{\itemsep}{0pt}
   \setlength{\topsep}{-.5\parskip}
   \usecounter{asslistcounter}}}
  {\end{list}}

\newenvironment{bulletlist}{
 \begin{list}
  {$\bullet$}
  {\setlength{\leftmargin}{18pt}
   \setlength{\rightmargin}{0pt}
   \setlength{\itemindent}{0pt}
   \setlength{\labelsep}{5pt}
   \setlength{\labelwidth}{13pt}
   \setlength{\listparindent}{\parindent}
   \setlength{\parsep}{0pt}
   \setlength{\itemsep}{0pt}
   \setlength{\topsep}{-.5\parskip}
   \usecounter{bulllistcounter}}}
 {\end{list}}

\usepackage{mathpazo}

\author[P.~Hamacher]{Paul Hamacher}

\address{Paul Hamacher\\
Technische Universit\"at M\"unchen\\
Zentrum Mathematik - M 11\
Boltzmannstra{\ss}e 3\\
85748 Garching\\
Deutschland}
\email{hamacher@ma.tum.de}

\def \an {{\mathrm {an}}}

\def \et {{\mathrm{\acute{e}t}}}

\DeclareMathOperator{\qIsol}{qIsol}
\DeclareMathOperator{\LocIsom}{LocIsom}
\DeclareMathOperator{\Tot}{Tot}
\DeclareMathOperator{\Top}{Top}

\def \L {{\mathrm{L}}}
\def \R {{\mathrm{R}}}
\def \D {{\mathrm{D}}}
\def \H {{\mathrm{H}}}

\title[Cohomology with compact support]{On the generalisation of cohomology with compact support to non-finite type schemes}

\begin{document}

 \begin{abstract}
 {In this article we extend Deligne's construction of Grothendieck's six operations on the derived category of torsion sheaves over the \'etale site of a scheme for morphisms of finite type to a larger class of morphisms. This class includes profinite \'etale coverings as well as separated morphisms perfectly of finite type}
 \end{abstract}
 
 \maketitle
 
\section*{Introduction}
 
 In the recent years there has been a trend in arithmetic algebraic geometry to work directly with geometric objects which are not of finite type. The most popular example of this are Scholze's perfectoid spaces, which form a family of ``infinite'' objects such that any analytic adic space over $\ZZ_p$ has a pro\'etale cover which is a perfectoid space. The most important instances of perfectoid spaces, from the view of the Langlands program, are (adic) infinite level Shimura varieties and their local analogues. For schemes this developement has taken place on a smaller scale; examples are the work about perfect schemes with application to the affine flag variety of Bhatt and Scholze (\cite{BhattScholze:AffGr}), as well as their work on pro\'etale morphisms (\cite{BhattScholze:ProEtTop}).
 
 While we have a definition of cohomology of \'etale cohomology with compact support for any adic space by Huber's work \cite{Huber:EtCohBook}, we can only define the \'etale cohomology group with compact support of a scheme (or more generally the direct image with compact support of a morphism) if it is of finite type. The aim of this paper is to extend the definition to a larger class of morphisms which we will call of finite expansion. This class will as include profinite \'etale coverings as well as perfections.

 Our construction follows the general framework of Deligne in \cite{SGA4.3}. Given a compactification of a separated morphism of finite extension, that is a a decomposition $f = p \circ j$ with $j$ an open immersion and $p$ separated universally closed, we define the higher direct image of compact support of $f$ as
 \[
  \R f_! \coloneqq \R p_\ast \circ j_!.
 \] 
Imposing some mild conditions on the source and target of $f$, this functor is well-defined. More precisely, we prove the following results.

 \begin{nothm}
  Let $S$ be a scheme and $\Ascr$ be a torsion sheaf on the \'etale site of $S$. Let $(S)$ denote the category whose objects are qcqs schemes over $S$ and whose morphisms are separated morphism of schemes of finite expansion.
  \begin{subenv}
   \item There is an essentially unique way to assign to every morphism $f\colon X \to Y$ a functor $\R f_!\colon \D(X,\Ascr_X) \to \D(Y,\Ascr_Y) $ such that $\R f_! = \R f_\ast$ if $f$ is universally closed, $\R f_! = f_!$ if $f$ is an open immersion and such there is a collection of isomorphisms $\R (f \circ g)_! \cong \R f_! \circ \R g_!$ satisfying the usual cocycle condition (see Theorem~\ref{thm-lowershriek} for details).
   \item The functor $\R f_!$ has a partial right adjoint $\R f^! \colon \D^+(Y,\Ascr_Y) \to \D^+(X,\Ascr_X)$. The formalism of Grothendieck's six operations for morphisms of finite type extends to $(S)$ (see sections~\ref{sect-bc} and \ref{sect-excinvim} for details).
  \end{subenv}
 \end{nothm} 

 For separated $\CC$-schemes of finite expansion the above notion of $\R f_!$ coincides with the topological direct image functor with proper support. More precisely, the following generalisation of Artin's comparison theorem holds.
 
 \begin{nothm}
  Let $f\colon X \to Y$ be a morphism of separated $\CC$-schemes of finite expansion, $\Fscr \in \D^+_{\rm tors}(X,\ZZ_X)$ and denote by $(-)_{\rm an}$ the analytification functor. Then there exists a natural isomorphism $(\R f_!\Fscr)_{\rm an} \cong \R f_{\an, !} \Fscr_{\rm an}$ (see section~\ref{sect-comparison-Artin} for details).
 \end{nothm} 

 We give a brief outline of the article. In sections~\ref{sect basic properties} and \ref{sect-ZMT}, we define morphisms of finite expansion and prove some of their properties; in particular we show that any separated morphism of finite expansion can be written as the composition of an open embedding and a universally closed morphism. In section~\ref{sect-univclosedbc}, we prove that the proper base change theorem extends to the case of separated universally closed morphisms, which is the main ingredient for the proof that $\R f_!$ does not depend on the choice of compactification. Then we use Deligne's formal construction from \cite[Exp.~XVII, \S~5]{SGA4.3} to put the pieces together and prove the well-definedness of $\R f_!$ in section~\ref{sect-lowershriek}. The formalism of Grothendieck's six operations can now deduced by the same arguments as in \cite[Exp.~XVII, XVIII]{SGA4.3}. We give brief sketch of the proofs in sections~\ref{sect-bc} and \ref{sect-excinvim} to convince the reader that Deligne's proofs still apply. We then show in section~\ref{sect-comparison-ad-hoc} that our definition coincides with the ad-hoc definitions used so far in the literature. Finally, in section~\ref{sect-comparison-Artin} we recall the analytification functor for complex schemes (which may not be of finite type) and extend Artin's comparison theorem for compact support cohomology to complex schemes of finite expansion.

{\bf Acknowledgements:} I am very grateful to Daniel Harrer for his helpful remarks and his explanations about derived categories. Furthermore I thank Timo Richarz, Ulrich Bauer and Kai Behrens for helpful discussions. The author was partially supported by the ERC Consolidator Grant "Newton strata".

\section{Morphisms of finite expansion}
In this section we introduce morphisms of finite expansion, which can be seen as relaxing the finiteness conditions of morphisms of finite type a bit. Building upon the analogous results for morphisms of finite type, we prove under a mild condition that they are compactifiable, i.e.\ can be written as  a composition of an open immersion and a universally closed separated morphism and study the behaviour of \'etale cohomology with respect to morphisms of finite expansion.

\subsection{Basic definitions and properties} \label{sect basic properties}
The key definition is:
 
\begin{definition}
 Let $R$ be a ring and $A$ be an $R$-algebra. A family $(a_i)_{i \in I}$ of elements in $A$ is called quasi-generating system of $A$, if $A$ is integral over $R[a_i \mid i\in I]$; the $a_i$ are called quasi-generators. If there exists a finite quasi-generating system of $A$, we say that $A$ is of finite expansion over $R$.
\end{definition}

\begin{remark}
 Alternatively we could say an $R$-algebra $A$ is of finite expansion if and only if there exists an integral morphism $R[x_1,\ldots, x_n] \to A$. In particular, the structure morphism decomposes into a morphism of finite presentation and an integral morphism.
\end{remark}

\begin{lemma}
 We fix a ring $R$ and an $R$-algebra $A$ be an $R$-algebra.
 \begin{subenv}
  \item Let $B$ be an $A$-algebra and assume that $A$ is of finite expansion over $R$. Then $B$ is of finite expansion over $R$ if and only if it is of finite expansion over $A$.
  \item Let $f_1,\ldots,f_m \in A$ such that $(f_1,\ldots,f_m)_A = A$. Then $A$ is of finite expansion over $R$ if and only if $A_{f_i}$ is of finite expansion over $R$ for every $i$.
  \item Let $R'$ be another $R$-algebra and assume that $A$ is of finite expansion over $R$. Then $R' \otimes_R A$ is finite expansion over $R'$.
  \item Let $R'$ be a faithfully flat $R$-algebra and assume that $R' \otimes A$ is of finite expansion over $R'$. Then $A$ is of finite expansion over $R$.
 \end{subenv}
\end{lemma}
\begin{proof}
  The ``only if'' part of the first assertion is obvious. To show the other direction, let $a_1,\ldots,a_m \in A$ and $b_1,\ldots,b_n \in B$ be quasi-generators over $R$ and $A$, respectively. Then $A[b_1,\ldots,b_n]$ is integral over $R[a_1,\ldots,a_m, b_1,\ldots,b_n]$, thus $B$ is integral over $R[a_1,\ldots,a_m, b_1,\ldots,b_n]$.

 The 'only if' part of the second assertion follows from the first part. Now if $A_{f_i}$ are of finite expansion over $R$, we choose finite systems of quasi-generators $b_{i,j} = a_{i,j}/f_i^{N_{i,j}}$. Denote by $B \subset A$ the subalgebra generated by all $a_{i,j}$ and $f_i$. Then $R[\{b_{i,j}\}_j] \subset B_{f_i}$ and thus $A_{f_i}$ is integral over $B_{f_i}$ for all $i$. Hence $A$ is integral over $B$.
 
 The third assertion follows as both properties, being of finite type and being integral, are preserved under tensor products. Now let $R \to R'$ be faithfully flat and fix a finite system of quasi-generators $m_i = \sum a_{i,j} \otimes r_{i,j}' $ of $R' \otimes_R A$ over $R'$. Let $B \subset A$ be the subalgebra generated by all $a_{i,j}$. Since $R'[\{m_i\}_i] \subseteq R'[\{1 \otimes a_{i,j}\}_{i,j}] = R' \otimes_R B$, the ring extension $R' \otimes_R B \subset R' \otimes_R A$ is integral. Thus $A$ is integral over $B$ by faithfully flat descent.
\end{proof}

\begin{cordef}
 We call a morphism $f\colon X \to Y$ of schemes locally of finite expansion if the following equivalent conditions are satisfied.
 \begin{assertionlist}
  \item For every affine open subscheme $V \subset Y$ and every affine open subscheme $U \subset f^{-1}(V)$, the $\Oscr_Y(V)$-algebra $\Oscr_X(U)$ is of finite expansion.
  \item There exists a covering $Y = \bigcup V_i$ by open affine subschemes $V_i \cong \Spec R_i$ and a covering $f^{-1}(V_i) = \bigcup U_{i,j}$ by open affine subschemes $U_{i,j} \cong \Spec A_{i,j}$ such that for all $i,j$ the $R_i$-algebra $A_{i,j}$ is of finite expansion.
 \end{assertionlist}
 We say that a morphism $f\colon X \to Y$ is of finite expansion if it is locally of finite expansion and quasi-compact.
\end{cordef}

\begin{corollary}
 \begin{subenv}
  \item The properties ``locally of finite expansion'' and ``of finite expansion'' of morphisms of schemes are stable under composition, base change and faithfully flat descent, and are local on the target. The property ``of finite expansion'' is also local on the source.
  \item Let $f\colon X \to Y$ and $g\colon Y \to Z$ be morphisms of schemes. If $g \circ f$ is locally of finite expansion (resp.\ of finite expansion and $g$ is quasi-separated), then $f$ is locally of finite expansion (resp.\ of finite expansion).
 \end{subenv}
\end{corollary}

A morphism of finite expansion behaves similarly as a morphisms of finite type. For example, we show below that the closed points of a scheme locally of finite expansion over a field are very dense and can be identified their residue field.

\begin{lemma} \label{lemma closed points}
 Let $k$ be a field, let $X$ be a $k$-scheme locally of finite expansion, and let $x \in X$. Then the following are equivalent.
 \begin{assertionlist}
  \item The point $x \in X$ is closed.
  \item The field extension $\kappa(x)/k$ is algebraic.
 \end{assertionlist}
\end{lemma}
\begin{proof}
 The implication $(b) \Rightarrow (a)$ holds for any $k$-scheme. To see the other direction, note that $\Spec \kappa(x) \to \Spec k$ is of finite expansion.  We choose a finite type $k$-subalgebra $A \subset \kappa(x)$ such that $\kappa(x)$ is integral over $A$. The latter property implies that $A$ is a also a field and hence $k \subset A$ is finite by Hilbert's Nullstellensatz. In particular, $\kappa(x)/k$ is algebraic.
\end{proof}

\begin{lemma}
 Let $f\colon X \to Y$ be a morphism of finite expansion. If $Y$ is Jacobson, so is $X$.
\end{lemma}
\begin{proof}
 Indeed, this can be checked locally, so we may assume that $X$ and $Y$ are affine. Now the claim follows since the property of being Jacobson is preserved under integral and finitely presented morphisms.
\end{proof}

 After imposing some mild conditions, we can also decompose a non-affine morphism of finite expansion into an integral morphism and a morphism of finite presentation. To prove this, we use that every separated morphism of qcqs schemes can be written as composition of an affine morphism and a morphism of finite presentation (\cite[Thm.~1.1.2]{Temkin:RiemannZariski}). The latter may even assumed to be proper by \cite[Thm.~1.1.3]{Temkin:RiemannZariski}. 
 
 \begin{proposition} \label{prop temkin decomposition}
 Let $f\colon X \to Y$ be a separated morphism of qcqs schemes.
 \begin{subenv}
  \item If $f$ is universally closed then it can be decomposed as $f = h \circ g$ with $g$ integral and $h$ proper.
  \item If $f$ is of finite expansion then it can be decomposed as $f = h \circ g$ with $g$ integral and $h$ of finite presentation.
  \end{subenv}
 \end{proposition}
 \begin{proof}
  For the first assertion we may consider any decomposition as in \cite[Thm.~1.1.3]{Temkin:RiemannZariski}. Indeed $g$ is affine by construction and universally closed by the cancellation property and hence integral.
  
  To show the second assertion, write $f = h \circ g$ with $g\colon X \to X_0$ affine and $h\colon X_0 \to Y$ of finite presentation. We write $g$ as limit of affine morphisms $g_\lambda\colon X_\lambda \to X_0$ of finite presentation. We fix a finite affine open covering $X_0 = \bigcup_{i=1}^n \Spec R_{0,i}$. Over $\Spec R_{0,i}$ the morphisms $g$ and $g_\lambda$ corresponds to a $R_{0,i}$-algebras $R_i$ of finite expansion and $R_{\lambda,i}$ of finite presentation, respectively. We fix a finite set $E_i \subset R_i$ of quasi-generators for each $i$.
  We have $R_i = \varinjlim R_{\lambda,i}$ by definition, thus we get $E_i \subset R_{\lambda,i}$ for all $i$ if $\lambda$ is big enough. Replacing $X_0$ be $X_\lambda$, we may thus assume that $g$ is integral. 
 \end{proof}
 
  \begin{corollary} \label{cor universally closed satisfies finiteness property}
  Every separated universally closed morphism between qcqs schemes is of finite expansion
 \end{corollary}
 \begin{proof}
  This is a direct consequence of the first assertion of the previous lemma.
 \end{proof}  

\subsection{Compactification of morphisms} \label{sect-ZMT}

The compactification of morphisms of finite expansion will proceed in two steps. First we compactify the morphisms of relative dimension zero, then we use Lemma~1.8 and the compactification of finitely presented morphisms to generalise this construction to arbitrary morphisms of finite expansion.

\begin{definition}
 Let $f\colon X \to Y$ be a morphism of schemes. We say that $f$ is locally quasi-profinite if it is locally of finite expansion and its fibres are totally disconnected. We call $f$ quasi-profinite if it is quasi-compact and locally quasi-profinite.
\end{definition}

There are several equivalent criterions to determine whether a morphism is locally quasi-profinite.
\begin{lemma}
 A morphism $f\colon X \to Y$ be a morphism of finite expansion. The following are equivalent
 \begin{assertionlist}
  \item $f$ is quasi-profinite.
  \item Every point $x \in X$ is a closed in the respective fibre $f^{-1}(f(x))$.
  \item For every $x \in X$ the extension of residue fields $\kappa(f(x))/\kappa(x)$ is algebraic.
 \end{assertionlist}
\end{lemma}
\begin{proof}
 The assertion (a) is equivalent to (b) since a scheme is zero-dimensional if and only if it is totally disconnected (see e.g.\ \cite[Tag 04MG]{AlgStackProj}). Now (b) and (c) are equivalent by Lemma~\ref{lemma closed points}   
\end{proof}

As a consequence, one easily checks that the following permanence properties hold.

\begin{proposition}
 \begin{subenv}
  \item The properties ``locally quasi-profinite'' and ``quasi-profinite'' of morphisms of schemes are stable under composition, base change and faithfully flat descent, and local on the target. The property ``quasi-profinite'' is also local on the source.
  \item Let $f\colon X \to Y$ and $g\colon Y \to Z$ be morphisms of schemes. If $g \circ f$ is locally quasi-profinite (resp.\ quasi-profinite and $g$ is quasi-separated), then $f$ is locally of quasi-profinite (resp.\ of quasi-profinite).
 \end{subenv}
\end{proposition}

In order to consider the property ``quasi-profinite'' locally at a point $x \in X$, we make the following definition.

\begin{definition}
 Let $X$ be a topological space. We say that $x \in X$ is quasi-isolated if $\{x\}$ is a connected component of $X$.
\end{definition}

\begin{lemma} \label{lemma quasi-isolated points}
 Let $k$ be a field, let $X$ be a separated $k$-scheme of finite expansion and let $x \in X$. We write $X= \varprojlim X_\lambda$, where $X_\lambda$ are of finite presentation over $k$ with finite transition morphisms (cf.\ Proposition~\ref{prop temkin decomposition}). Then the following are equivalent.
 \begin{assertionlist}
  \item $x$ is quasi-isolated.
  \item $\{x\}$ is an irreducible component of $X$.
  \item $k \to \Oscr_{X,x}$ is integral.
  \item The  image $x_\lambda$ of $x$ in $X_\lambda$ is isolated for some $\lambda$. In this case $x_{\lambda'}$ is isolated for all $\lambda' > \lambda$.
 \end{assertionlist}
\end{lemma}
\begin{proof}
 Obviously, the first assertion implies the second assertion. If $\{x\}$ is an irreducible component, $k \to \kappa(x)$ is integral by Lemma~\ref{lemma closed points} and $\kappa(x) = \Oscr_{X,x}^{\rm red}$ since $x$ is a generic point, proving that $k \to \Oscr_{X,x}$ is integral.
 
  Now assume that $\Oscr_{X,x}$ is integral over $k$. We denote $U_\lambda = \Spec \Oscr_{X_\lambda,x_\lambda} \setminus \{x_\lambda\}$. Since $\{x\} = \Spec \Oscr_{X,x} = \varprojlim  \Spec \Oscr_{X_\lambda,x_\lambda}$, we get $\varprojlim U_\lambda= \emptyset$. By \cite[Tag 01Z2]{AlgStackProj} this implies that there exists a $\lambda$ with $U_\lambda = \emptyset$. Thus $\Spec \Oscr_{X_\lambda,x_\lambda} = \{x_\lambda\}$, hence $x_\lambda$ is isolated.
  
  On the other hand assume that $x_\lambda$ is isolated. Then the connected component $C$ of $X$ containing $x$ is contained in the fibre of over $x_\lambda$. Since the fibres of $X \to X_\lambda$ are totally disconnected by lying over, we must have $C = \{x\}$. The same argument also proves that $x_{\lambda'}$ is isolated for all $\lambda' > \lambda$.
\end{proof}

Now we can deduce a generalisation of the algebraic version of Zariski's Main Theorem. Fix a ring $R$ and an $R$-algebra $A$ of finite expansion. We denote by
\[
 \qIsol(A/R) \coloneqq \{\qfr \in \Spec A \mid \qfr \textnormal{ is quasi-isolated in its fibre over } \Spec R\}.  
\]
 Moreover, we denote by $\LocIsom(A/R)$ the set of prime ideals $\qfr \subset A$ such that there exists $s \in R$ with $s \not\in \qfr$ and $R_s = A_s$. In other words, if $X=\Spec A$ and $Y=\Spec R$ and $f\colon X \to Y$ is the morphism corresponding to $\varphi$, then $\LocIsom(A/R)$ is the set of points $x\in X$ such that there exists an open neighbourhood $V \subset Y$ of $f(x)$ such that $f$ induces an isomorphism $f^{-1}(V) \isom V$. 

\begin{proposition} 
 Let $A$ be an $R$-algebra of finite expansion and $R'$ the integral closure of $R$ in $A$. Then
 \[
  \qIsol(A/R) = \LocIsom(A/R');
 \]
 in particular $\qIsol(A/R)$ is open.
\end{proposition}
\begin{proof}
 Since $\LocIsom(A/R') \subseteq \qIsol(A/R) \subseteq \qIsol(A/R')$, it suffices to show $\qIsol(A/R') \subseteq \LocIsom(A/R')$. We write $A = \varinjlim A_\lambda$ as limit of finite type $R'$-subalgebras with finite transition maps. In particular, we have $\qIsol(A_\lambda/R') = \LocIsom(A_\lambda/R')$ by the affine version of Zariski's main theorem.
  
 Denote by $f_\lambda\colon \Spec A \to\Spec A_\lambda$, $g_\lambda\colon \Spec A_\lambda \to \Spec R$ and $f\colon  \Spec A \to \Spec R'$ the canonical morphisms. Then we have for every $\pfr \in \qIsol(A/R')$ that $f_\lambda(\pfr) \in \qIsol(A_\lambda/R') = \LocIsom(A_\lambda/R') $ for $\lambda$ big enough by Lemma~\ref{lemma quasi-isolated points}. Fixing such a $\lambda$, $V \coloneqq g_\lambda (\LocIsom(A_\lambda/R'))$ is an open neighbourhood of $f(\pfr)$ such that $g_{\lambda'}^{-1} (V) \subset \LocIsom(A_\lambda'/R')$ for every $\lambda' > \lambda$ by Lemma~\ref{lemma quasi-isolated points}~(d). Taking the limit we obtain that $f$ induces an isomorphism $f^{-1}(V) \isom V$. 
\end{proof}

Since the relative normalisation is local on the base, we may globalise the result as follows.
 
 \begin{proposition}\label{prop ZMT algebraic}
  Let $f\colon X \to Y$ be an affine morphism of finite expansion. Let $f'\colon X \to X'$ be the normalisation of $X$ relative to $Y$. Then there exists an open subscheme $U' \subset X'$ such that
  \begin{assertionlist}
   \item $f'^{-1}(U') \to U'$ is an isomorphism and
   \item $f'^{-1}(U') \subseteq X$ is the set of points which are quasi-isolated in the fibre of $f$.
  \end{assertionlist}
  In particular, if $f$ is quasi-profinite then $f'$ is an open immersion.
 \end{proposition}

  We can now deduce that any separated morphism of finite expansion between qcqs schemes is compactifiable, in the sense that it can be written as the composition of an open embedding with a universally closed morphism. Note that the converse also holds true by Corollary~\ref{cor universally closed satisfies finiteness property}.
 
 \begin{theorem}
  Let $f\colon X \to Y$ be a separated morphism of finite expansion between qcqs schemes. Then $f$ can be written as composition $f = \fbar \circ j$ where $j\colon X \to \Xbar$ is an open embedding and $\fbar\colon \Xbar \to Y$ is separated and universally closed.
 \end{theorem}
 \begin{proof}
  Write $f = g \circ h$ with $g$ is of finitely presented and $h$ integral (cf.~Proposition~\ref{prop temkin decomposition}). We can write $g = \gbar \circ j'$ where $j'$ is an open embedding and $\gbar$ is proper (see e.g.~\cite[Thm.~4.1]{Conrad:NagataDeligne}) . Note that we may assume that $j'$ is affine after blowing up along its complement. We denote $h' = j' \circ g$, then $f = \gbar \circ h'$, where $\gbar$ is proper and $h'$ is affine and quasi-profinite. By Proposition~\ref{prop ZMT algebraic} we can write $h' = \hbar \circ j$, where $\hbar$ is integral and $j$ an open embedding. Setting $\fbar = \gbar \circ \hbar$, the claim follows.
 \end{proof}

\subsection{\'Etale cohomology for universally closed morphisms} \label{sect-univclosedbc}

When working with \'etale cohomology, we will mostly assume that we are in the following setup. We fix a base scheme $S$ and an \'etale ring sheaf $\Ascr$ on $S$. For any $S$-scheme $X$, we denote by $\Ascr_X$ the pull-back of $\Ascr$ to $X$ and by $\D(X, \Ascr_X)$ the derived category of $\Ascr_X$-modules. We denote by $\D_{\rm tors} (X,\Ascr_X) \subseteq \D (X,\Ascr_X)$ the full subcategory of complexes whose cohomology group are torsion, or equivalently the derived category of torsion $\Ascr_X$-modules. Often we will directly assume that $\Ascr$ is torsion, thus $\D_{\rm tors} (X,\Ascr_X) =\D (X,\Ascr_X)$.


\begin{proposition}[Universally closed base change] \label{prop proper BC}
 Let
 \begin{center}
  \begin{tikzcd}
   X' \arrow{r}{g'} \arrow{d}{f'} & X \arrow{d}{f} \\
   S' \arrow{r}{g} & S
  \end{tikzcd}
 \end{center}
 be a Cartesian diagram with $f$ separated and universally closed. Let $\Fscr$ be an Abelian torsion sheaf on the \'etale site of $X$. Then the base change morphism $\varphi^q\colon g^\ast \R^qf_\ast\Fscr \to (\R^qf'_\ast)g'^\ast\Fscr$ is an isomorphism for every $q \geq 0$.
\end{proposition}
\begin{proof}

 By the same argument as in \cite[Exp.~XII,~Lemme~6.1]{SGA4.3} it suffices to consider the case that $S = \Spec R$ is the spectrum of a strictly Henselian local ring and $g'$ is the embedding of its closed point. In particular, we can write $f = p \circ f_0$ with $p$ proper and $f_0$ integral. By a standard application of Leray's spectral sequence, it suffices to prove the base change theorem for $f_0$ and $p$ instead of $f$. The result for $p$ is the classical base change theorem; thus it suffices to consider the case that $f$ is integral.
 
 Altogether, we reduced to the case of an integral morphism $f\colon \Spec A \to \Spec R$ with $R$ strictly Henselian and $g\colon \Spec R/\mfr \to \Spec R$ is the embedding of a closed point. Then $\varphi^q$ can be identified with the canonical morphism
 \[
  \H^q(\Spec A, \Fscr) \to \H^q(\Spec A/\mfr A, g'^\ast\Fscr).
 \]
 Since $A$ is integral over $R$, the couple $(A,\mfr A)$ is a Henselian pair. Thus the assertion is proven in \cite[Thm.~1]{Gabber:ProperBC}.
\end{proof}

\begin{corollary}
 Let $k$ be a separably closed field and let $X$ be a separated, universally closed scheme over $k$. Then for any separably closed field extension $K/k$ and any torsion sheaf $\Fscr$ on $X$ we have $\H^i(X_K,\Fscr) = \H^i(X,\Fscr)$ for any $i \in \ZZ$. 
\end{corollary}

Using the language of \cite[Exp.~XVII,~\S~4]{SGA4.3}, we may reformulate the above proposition as follows.

\begin{proposition}[cf.~{\cite[Exp.~XVII,~Thm.~4.3.1]{SGA4.3}}] \label{prop proper BC 2}
  Let
 \begin{center}
  \begin{tikzcd}
   X' \arrow{r}{g'} \arrow{d}{f'} & X \arrow{d}{f} \\
   S' \arrow{r}{g} & S
  \end{tikzcd}
 \end{center}
 be a Cartesian diagram with $f$ separated and universally closed. Let $\Ascr$ and $\Ascr'$ be \'etale ring sheaves on $S$ and $S'$, respectively, and let $L$ be a complex of torsion ($\Ascr',g^\ast\Ascr$)-bimodules which is bounded from above. We assume that the fibre dimension of $f$ is bounded. Then the base change morphism of functors from $\D^-(X,f^\ast \Ascr) \to \D^-(S',\Ascr')$
 \[
  \varphi_K \colon L \otimes^{\L} \L g^\ast \R f_\ast K \to \R f'_\ast (f'^\ast L \otimes^\L \L g'^{\ast} K)
 \] 
 defined in (4.2.13.2) of \cite[Exp.~XVII]{SGA4.3} restricts to an isomorphism $\D_{\rm tors}^-(X,f^\ast \Ascr) \to \D_{\rm tors}^-(S',\Ascr')$.
\end{proposition}
\begin{proof}
 The proof is identical to the one of \cite[Exp.~XVII,~Thm.~4.3.1]{SGA4.3} after one replaces the reference to the proper base change theorem with Proposition~\ref{prop proper BC}.
\end{proof}

As in the case of proper morphisms, we obtain an upper bound on the cohomological dimension of $\R f_\ast$ for $f$ universally closed.

\begin{lemma} \label{lemma cd of higher direct image} 
  Let $f\colon X \to S$ be a separated universally closed morphism of qcqs schemes such that all fibres of $f$ have at most dimension $d$. Then $\R f_\ast$ is of cohomological dimension $\leq 2d$.  
\end{lemma}
\begin{proof}
 Using Proposition~\ref{prop temkin decomposition} and Leray's spectral sequence, it suffices to consider the two cases where $f$ is proper and where $f$ is integral. Both cases are well known (see for example \cite[Tags 095U, 04C2]{AlgStackProj}).
\end{proof}

\section{Grothendieck's six operations}

The arguments in the construction of $\R f_!$ and $\R f^!$ for a separated morphism of finite expansion $f\colon X \to Y$ are mostly identical to Deligne's construction in \cite{SGA4.3} for finite type morphisms, where the assertions in section~\ref{sect-univclosedbc} substitute their finite type analogues. For the reader's convenience (and in order to convince him that we can indeed generalise the construction), we will nevertheless repeat some of the arguments of \cite{SGA4.3} rather than simply refer to them.

\subsection{Higher direct images with compact support.} \label{sect-lowershriek}

We fix a base scheme $S$ and an \'etale torsion ring sheaf $\Ascr$ on $S$. Denote by $(S)$ the subcategory of the category of qcqs $S$-schemes whose objects are the same and whose morphisms are the separated morphisms of finite expansion. Similarly, let $(S,p)$ and $(S,i)$ denote the subcategory of qcqs $S$-schemes with separated universally closed $S$-morphisms and open $S$-embeddings, respectively. 

Assume $f\colon X \to Y$ is a morphism in either $(S,p)$ or $(S,i)$. In these cases we define the push-forward with compact support $\R f_!\colon \D(X,\Ascr_X) \to \D(Y,\Ascr_Y)$ as follows. If $f$ is a morphism in $(S,p)$, we define $\R f_! \coloneqq \R f_\ast$; if $f$ is in $(S,i)$, we define $\R f_! = f_!$ as the extension by zero. We would like to extend these definition all morphisms in $(S)$ by composition.

In order to prove that $\R f_!$ is well-defined for a general morphism $f$ of finite expansion,  we consider a commutative diagram in $(S)$
 \begin{equation} \label{diag compatibility}
  \begin{tikzcd}
   U \arrow[hook]{r}{j'} \arrow{d}{p'} & X \arrow{d}{p} \\
   V \arrow[hook]{r}{j} & Y
  \end{tikzcd}
 \end{equation}
where the horizontal morphisms are in $(S,i)$ and the vertical morphisms are in $(S,p)$. Following \cite[Exp.~XVII,~5.1.5]{SGA4.3}, we construct a natural comparison morphism $d\colon j_{!} \R p'_\ast K \to \R p_\ast j'_! K$. Since $j_!$ is left adjoint to $j^\ast$, constructing $d$ is the same as giving a morphism $d'\colon \R p'_\ast K \to j^\ast \R p_\ast j'_! K$. We extend the above diagram by $X_0 \coloneqq X \times_Y V$ to
  \begin{center}
  \begin{tikzcd}
   U \arrow[hook]{r}{k} \arrow{d}{p'} & X_0 \arrow{d}{p_0} \arrow[hook]{r}{j_0} & X \arrow{d}{p} \\
   V \arrow[equal]{r} & V \arrow[hook]{r}{j} & Y.
  \end{tikzcd}
 \end{center}
We get $j^\ast \R p_\ast j'_! K \cong \R p_{0\ast} j_0^\ast j'_! K \cong \R p_{0\ast} k_! K$. Since $k$ is an open and closed immersion, we have $k_\ast = k_!$ and thus we obtain $\R p_{0\ast} k_! K = \R p_{0\ast} k_\ast K \cong \R p'_\ast K$. Now we define $d'$ as the inverse of this chain of isomorphisms and thus obtain $d$.

\begin{lemma}[cf. {\cite[Exp.~XVII,~Lemme~5.1.6]{SGA4.3}}]
 The morphism $d$ is an isomorphism.
\end{lemma}
\begin{proof}
 It remains to show that the restriction of $d$ to $Z \coloneqq Y\setminus V$ is an isomorphism. Since $\restr{j_{!} \R p'_\ast K}{Z} = 0$, this assertion is equivalent to $ \restr{\R p_\ast j'_! K}{Z} = 0$. For this we apply Proposition~\ref{prop proper BC 2} to the Cartesian square
 \begin{center}
  \begin{tikzcd}
   X'' \arrow{r}{i'} \arrow{d}{p''} & X  \arrow{d}{p} \\
   Z  \arrow{r}{i} & Y
  \end{tikzcd}
 \end{center}
 for $\Ascr' = i^\ast \Ascr$ and $L = \Ascr'[0]$. We obtain $i^\ast \R p_\ast j'_! K \cong  \R p''_\ast i'^\ast j'_! K = 0$, proving the lemma.
\end{proof}

We can now apply the theory developed in \cite[Exp.~XVII,~\S~3]{SGA4.3} to obtain a well-defined functor $\R f_! \colon \D (X,\Ascr_X) \to \D(Y,\Ascr_Y)$ for every morphism $f\colon X \to Y$ in $(S)$. More precisely, the following holds.

\begin{theorem}[cf.~{\cite[Exp.~XVII,~Thm.~5.1.8]{SGA4.3}}] \label{thm-lowershriek}
 One can assign to every morphism $f$ in the category $(S)$ a functor
 \[
  \R f_! \colon \D (X,\Ascr_X) \to \D (Y,\Ascr_Y)
 \]
 and to every composition $f = g \circ h$ an isomorphism $c_{g,h} \colon \R f_! \to \R g_! \circ \R h_!$ in an essentially unique way such that the following conditions are satisfied.
 \begin{assertionlist}
  \item[(i)] The assignment $(f,g) \mapsto c_{f,g}$ satisfies the cocycle condition
  \[ 
   (c_{f,g} \ast \R h_!) \circ c_{fg,h} = (\R f_! \ast c_{g,h}) \circ c_{f,gh}.
  \]
  \item[(ii)] If $f$ is a universally closed morphism (or an open immersion, respectively) $\R f_!$ coincides with the definition above.
  \item[(iii)] For every diagram (\ref{diag compatibility}), the morphism $d$ equals $c_{p,j'} \circ c_{j,p'}^{-1}$
 \end{assertionlist}
 Moreover, the functors $\R f_!$ are triangulated and for every $n \in \ZZ$ there exists an $m \in \ZZ$ such that if $K \in \D(X,\Ascr_X)$ with $\H^i(K) = 0$ for every $i < n$ then $\H^j(\R f_! K) = 0$ for every $j<m$. In particular $\R f_!$ restricts to a functor $\D^+ (X,\Ascr_X) \to \D^+ (Y,\Ascr_Y)$. 
\end{theorem}
\begin{proof}
 This is an application of \cite[Exp.~XVII,~Prop.~3.3.2]{SGA4.3} with the following data:
 \begin{bulletlist}
  \item $F(X) \coloneqq \D (X,\Ascr_X)$.
  \item For any morphism $f$ in $(S,p)$ (and $(S,i)$, respectively) $f_\ast \coloneqq \R f_!$, as defined above.
  \item The transitivity isomorphisms for morphisms in $(S,p)$ (and $(S,i)$, respectively) are the usual ones.
  \item For every commutative diagram (\ref{diag compatibility}), we define the isomorphism $d$ as above.
 \end{bulletlist}
\end{proof}

\begin{remark}
 The construction of $d$ and the proof of Lemma~2.2 does not need that $\Ascr$ is torsion, as long as all functors are defined and the cohomology groups of $K$ are torsion. Thus we could also set $F(X) \coloneqq \D^+_{\rm tors} (X,\ZZ_X)$ in the proof of above theorem to obtain a family of functors
 \[
  \R f_! \colon \D_{\rm tors}^+ (X,\ZZ_X) \to  \D_{\rm tors}^+(Y,\ZZ_Y)
 \] 
 Note that for any $S$-morphism $X \to Y$ (higher) direct images and extension by zero commute with the forgetful functor $\D(X, \Ascr_X) \to \D(X, \ZZ_X)$ and $\D(Y,\Ascr_Y) \to \D(Y, \ZZ_Y)$ , respectively; thus the uniqueness of $\R f_!$ tells us that the diagram
 \begin{center}
  \begin{tikzcd}
   \D^+(X,\Ascr_X) \arrow{r} \arrow{d}{\R f_!} & \D^+_{\rm tors}(X,\ZZ_X) \arrow{d}{\R f_!} \\
   \D^+(Y,\Ascr_Y) \arrow{r} & \D^+_{\rm tors}(X,\ZZ_Y)
  \end{tikzcd}
 \end{center}
 commutes. In other words, the construction of $\R f_!$ is independent of $\Ascr$.
\end{remark}

\begin{definition}
  Let $f\colon X\to Y$ a morphism in $(S)$. We call the functor $\R f_!$ as constructed above the higher direct image with proper support of $f$.  We define the $q$-th direct image with proper support $\R^q f_!$ as the composition of $\R f_!$ with the $q$-th cohomology functor. In the special case where $Y= \Spec k$ for a separably closed field $k$, we may also write $\H_c^q(X,-) \coloneqq \R^q f_!$ and call it the $q$-th cohomology group with compact support. 
\end{definition}

As a direct corollary of Lemma~\ref{lemma cd of higher direct image}, we obtain the finiteness of the cohomological dimension of $\R f_!$.

\begin{lemma} \label{lemma cd of direct image with proper support}
 Let $f\colon X \to Y$ be a morphism in $(S)$ of relative dimension $\leq d$. Then $\R f_!$ is of cohomological dimension $\leq 2d$. In particular, if $X$ is of finite expansion over a separably closed field $k$, then $\H_c^i(X,\Fscr) = 0$ for any \'etale torsion sheaf $\Fscr$ and $i> 2\dim X$.
\end{lemma}

\subsection{Base change} \label{sect-bc}

In order to study the behavious of $\R f_!$ with respect to base change, let $g\colon S \to S'$ be a morphism of qcqs schemes and let $\Ascr$ and $\Ascr'$ be \'etale torsion ring sheaves on $S$ and $S'$, respectively. We fix a morphism of finite expansion $f\colon X \to S$; we choose a decomposition $f = p \circ j$ with $j\colon X \mono \Xbar$ an open immersion and $p\colon \Xbar \to S$ separated universally closed. Now fix a morphism $g\colon S' \to S$ and denote by $X',\Xbar',f',g',\gbar,j',p'$ the respective pullbacks, i.e.\ we have a commutative diagram
\begin{center}
 \begin{tikzcd}
  X' \arrow{rrr}{g'} \arrow[hook]{dr}{j'} \arrow{dd}{f'} & & & X \arrow[hook]{rd}{j} \arrow[crossing over]{dd}{f} & \\
  & \Xbar' \arrow{dl}{p'} \arrow{rrr}{\gbar} & & & \Xbar \arrow{dl}{p} \\
  S' \arrow{rrr}{g} & & & S &
 \end{tikzcd}
\end{center}
where all rectangles are Cartesian. The base change morphism will be defined as an isomorphism of the functors $\D^- (X,\Ascr_X) \to \D^-(S',\Ascr'), K \mapsto L \otimes^{\L} g^\ast \R f_! K$ and $K \mapsto \R f_!(f'^\ast L \otimes^{\L} g'^\ast K)$, where  $L$ is a bounded above complex of $(\Ascr',f^\ast\Ascr)$-bimodules over $S'$. For this we first we consider the baby case where $f = j$ and $p = \id$. The composite of the canonical isomorphisms $g'^\ast j'_! \isom j_! g^\ast$ and $L \otimes^{\L}_{g^\ast \Ascr} j'_! (-) \isom j'_!(j'^\ast  L \otimes^{\L}_{f^\ast \Ascr} (-))$ (for details see \cite[Exp.~XVII~5.2.1]{SGA4.3}) yields a natural isomorphism
\begin{equation} \label{eq BC special}
 L \otimes^{\L}_{g^\ast \Ascr} g^\ast j_! K \isom j'_!(j'^\ast  L \otimes^{\L}_{f^\ast \Ascr} g'^\ast K) .
\end{equation}
In the general case define the base change isomorphism  as the composite
\begin{equation} \label{eq BC}
 L \otimes^{\L} g^\ast \R p_\ast j_! K  \xrightarrow[Prop.~\ref{prop proper BC 2}]{\sim} \R p'_\ast (p'^\ast L \otimes^{\L} \gbar^\ast j_! K)
  \xrightarrow[(\ref{eq BC special})]{\sim} \R p'_\ast j'_! (f'^\ast L \otimes^{\L} \gbar^\ast K).
\end{equation}
Following the argumentation in \cite[Exp.~XVII~\S~5.2]{SGA4.3}, we obtain the following compatibility results.

\begin{lemma}
 In the above setting, the following assertions hold.
 \begin{subenv}
 \item The isomorphism (\ref{eq BC}) does not depend on the choice of $j$ and $p$.
 \item If $f = u \circ v$ where $u\colon Y \to S$ and $v \colon X \to Y$ are of finite expansion, the base change morphisms with respect to $u,v$ and $f$ fit inside a commutative diagram
 \begin{center}
  \begin{tikzcd}[column sep = small]
   L \otimes^{\L} g^\ast \R f_! K \arrow{rr}{\sim} \arrow[equal]{d}& & \R f_! (f'^\ast L \otimes^\L g''^\ast K) \arrow[equal]{d} \\
   L \otimes^{\L} g^\ast \R u_! \R v_! K \arrow{r}{\sim} & \R u'_! (u'^\ast L \otimes^{\L} g'^\ast \R v_! K) \arrow{r}{\sim} & \R u'_! \R v'_! (v'^\ast u'^\ast L \otimes^{\L} g''^\ast K)
  \end{tikzcd}
 \end{center}
 \item Assume we have morphisms of ringed schemes $(S'',\Ascr'') \xrightarrow{r}{g_2} (S',\Ascr') \xrightarrow{r}{g_1} (S,\Ascr)$ and $f\colon X \to S$ is a morphism of finite expansion. We consider the Cartesian diagram
 \begin{center}
  \begin{tikzcd}
   X'' \arrow{r}{g_2'} \arrow{d}{f'} & X' \arrow{r}{g_1'} \arrow{d}{f'} & X \arrow{d}{f} \\
   S'' \arrow{r}{g_2} & S' \arrow{r}{g_1} & S.
  \end{tikzcd}
 \end{center}
 Then the base change morphisms with respect to $g_1,g_2$ and $g = g_1 \circ g_2$ fit inside a commutative diagram
 \begin{center}
  \begin{tikzcd}
   \L g^\ast \R f_! \arrow{rr}{\sim} \arrow[equal]{d} & & \R f''_\ast g'^\ast \arrow{d}{\sim}\\
   \L g_2^\ast \L g_1^\ast \R f_! \arrow{r}{\sim} &\L g_2^\ast \R f'_! \L g_1'^\ast \arrow{r}{\sim} & \R f''_! \L g'^\ast .
  \end{tikzcd}
 \end{center}
 \end{subenv}
 \end{lemma}
 
\begin{proof}
 The proof is analogous to the proofs of \cite[Exp.~XVII~Lemme~5.2.3-5]{SGA4.3}, which are purely formal consequences of the constructions above.
\end{proof}
 
We may reformulate the above result as follows.

\begin{proposition} \label{prop BC isom}
 There is a natural isomorphism $L \otimes^{\L} g^\ast \R f_!K \isom \R f_!(f'^\ast L \otimes^{\L} g'^\ast K)$ of functors $\D^-(X,\Ascr_X) \to \D^-(S',\Ascr')$ which is compatible with composition.
\end{proposition} 
 
\begin{corollary}[Projection formula]
 Let $S$ be qcqs and $f\colon X \to S$ be a morphism of finite expansion. For $K \in \D^- (X,\Ascr_X)$ and $L \in \D^- (S,\Ascr)$ there exists natural isomorphism $ \R f_!K \otimes_{\Ascr} L \cong \R f_! (f^\ast K \otimes L)$
\end{corollary}
\begin{proof}
 This is the special case $Y=S$ and $g=\id_S$.
\end{proof} 
 
\begin{corollary}[Base change formula]
 Let 
 \begin{center}
  \begin{tikzcd}
   X' \arrow{d}{f'} \arrow{r}{g'} & X \arrow{d}{f} \\
   S' \arrow{r}{g} & S
  \end{tikzcd}
 \end{center}
 be a Cartesian square as above. Then there exists a canonical isomorphism $g^\ast \R f_! K \isom \R f'_! g'^\ast K$.
\end{corollary} 
\begin{proof}
 We apply the above proposition to $L = \Ascr_{S'}$.
\end{proof} 
 
\begin{corollary}
 Let $S$ be qcqs and $f\colon X \to S$ be of finite expansion. For any geometric point $\ybar$ of $Y$ and $q \in \NN_0$ there exists a canonical isomorphism $(\R^q f_! \Fscr)_{\ybar} \cong \H^q_c(X_{\ybar},\Fscr_{\ybar})$.
\end{corollary}
\begin{proof}
 We apply the above proposition to $S' = \ybar$.
\end{proof}

\begin{corollary}
 Let $X$ be a scheme of finite expansion over an separably closed field $k$ and $\Fscr$ be an \'etale torsion sheaf on $X$. Then for any separably closed field extension $k \subset K$ and any $q\in \NN_0$, there is a natural isomorphism $\H^q_c(X,\Fscr) \cong \H^q_c(X_K,\Fscr)$.
\end{corollary}

 Now let $f\colon X \to S$, $g\colon Y \to S$ be morphisms of finite expansion of qcqs schemes and denote by $h\colon P = X \times_S Y \to S $ their fibre product over $S$. With $p\colon P \to X$ and $q\colon P \to Y$ the canonical projection, we denote for $K \in \D^-(X,\Ascr_X), L \in \D^-(Y,\Ascr_Y)$
\[
 K \boxtimes_{\Ascr_P}^{\L} L \coloneqq p^\ast K \otimes^{\L}_{\Ascr} q^\ast L.
\]
As a consequence of the base change isomorphism we obtain the K\"unneth formula.

\begin{proposition}[K\"unneth formula]
In the situation above there exists a natural isomorphism
\[
 \R f_!K \otimes_{\Ascr}^{\L} \R g_!L \isom \R h_! (K \boxtimes_{\Ascr_P}^{\L} L)
\]
\end{proposition}
\begin{proof}
  Applying the base change formula twice, we obtain
 \[
  \R f_!K \otimes_{\Ascr}^{\L} \R g_!L \isom \R f_! (K \otimes_{\Ascr_X}^{\L} f^\ast \R g_! L) \isom \R f_! \R p_!(p^\ast K \otimes^{\L} q^\ast L).
 \]
 Since $\R h_! = \R f_! \circ \R p_!$, the right hand side equals $\R h_! (K \boxtimes_{\Ascr_P}^{\L} L)$.
\end{proof}

\begin{remark}
 Alternatively, we can also construct the isomorphism of the K\"unneth formula following \cite[XVII~\S~5.4]{SGA4.3}. More precisely, the general site-theoretic construction in \cite[XVII~5.4.1]{SGA4.3} yields a morphism
\[
 \R p_\ast (j_! K) \otimes_{\Ascr}^{\L} \R q_\ast (k_! L) \to \R r_\ast (j_! K \boxtimes_{\Ascr_{\Pbar}}^{\L} k_! L )
\]
where $f = p \circ j, g = q \circ k$ and $h = r \circ k$ are decomposions into a separated universally closed morphism and an open embedding. Moreover, by applying (\ref{eq BC special}) twice, we obtain an isomorphism
\[
 j_! K \boxtimes_{\Ascr_{\Pbar}}^{\L} k_! L \isom l_! (K \boxtimes_{\Ascr_P}^\L L).
\]
Their composition yield a morphism $\R f_!K \otimes_{\Ascr}^{\L} \R g_!L \to \R h_! (K \boxtimes_{\Ascr_P}^{\L} L)$, which is precisely the K\"unneth isomorphism by \cite[XVII~Lemme~5.4.3.5]{SGA4.3}. Note that the proof of the latter result is purely site-theoretic - first one uses (\ref{eq BC special}) to reduce to the case that $f$ and $g$ are proper (or separated universally closed, respectively), so that one can replace all direct images with proper support are simply direct images. Now the claim follows formally from the construction in \cite[XVII 5.4.1]{SGA4.3}.
\end{remark} 

By a similar argument as in the remark above, we may use \cite[XVII Lemme 5.4.3.2]{SGA4.3} to deduce the following slightly more general version.

\begin{corollary}
 Let $X_0$ and $Y_0$ be qcqs $S$-schemes and $f\colon X \to X_0$, $g \colon Y \to Y_0$ be $S$-morphisms of finite expansion and $K \in \D^-(X,\Ascr_X), L \in \D^-(Y,\Ascr_Y)$. We denote $h =(f,g)\colon X \times_S Y \to X_0 \times_S Y_0$. Then there exists a natural morphism
 \[
  \R f_!K \boxtimes^{\L} \R g_!L \isom \R h_! (K \boxtimes^{\L} L).
 \]
\end{corollary}

\subsection{Exceptional inverse image} \label{sect-excinvim}

The construction of $\R f^!$ is completely analogous to the finite type case in \cite[Exp.~XVIII,~\S~3]{SGA4.3}, so we merely give a short sketch of the arguments and refer the reader to corresponding parts of \cite{SGA4.3}. 

Let $f \colon X \to S$ be a morphism of finite expansion of qcqs schemes. For a torsion sheaf $\Fscr$ we denote by $C_\ell^\ast(\Fscr)$ the canonical modified flasque resolution of $\Fscr$ as defined in \cite[XVIII~3.1.2]{SGA4.3}. In particular, $C_\ell^\ast$ is functorial in $\Fscr$, commutes with localisation and inductive limits; and the functors $C_\ell^n$ are exact. 

Our aim is to construct a partial right adjoint $\R f^!\colon \D^+(S,\Ascr) \to \D^+(X,\Ascr_X)$ of $\R f_!$. As usual, we cannot use the adjoint functor theorem directly, since $\D^+(\cdot)$ is not cocomplete. To construct $\R f^!$, we fix an integer $d$ such that $f$ is of relative dimension $<d$ and a decomposition $f = p \circ j$, with $p$ separated universally closed and $j$ an open immersion. Let $\Fscr$ be an torsion $\Ascr_X$-module and $\Fscr^i$ the components of $\tau_{\leq 2d} C_\ell^\ast j_! \Fscr$ where $\tau_{\leq 2d}$ denotes the (canonical) truncation functor. Then
\begin{equation} \label{eq vanishing}
 \R^k p_\ast \Fscr^i = 0
\end{equation}
for $i>0$ since for $i \neq 2d$ the sheaf $\Fcal_i$ is flasque and for $i = 2d$ the claim follows from Lemma~\ref{lemma cd of direct image with proper support}. We define the functor $f_!^\bullet$ from the category of $\Ascr_X$-modules to complexes of $\Ascr$-modules by
\[
 f_!^\bullet(\Fscr) = p_\ast \tau_{\leq 2d} (C_\ell^\ast j_! \Fscr)
\]
Analogous to \cite[XVII~Lemme~3.1.4.8]{SGA4.3}, we see that this functor has the following properties. Its cohomology groups are supported in the interval $[0,2d]$, and the functors $f^i_!$ are exact and commute with filtered colimits. In particular the right derived functor $\R f_!^\bullet\colon \D(X,\Ascr_X) \to \D(S,\Ascr_S)$ exists. Moreover, for any complex $K$ of $\Ascr_X$-sheaves the composition
\[
 \R f_!^\bullet K = \Tot p_\ast \tau_{\leq 2d} C_\ell^\ast j_! K \to  p_\ast C_\ell^\ast j_! K = \R p_\ast j_! K = \R f_!K
\]
is an isomorphism by (\ref{eq vanishing}) and the usual spectral sequence for double complexes; thus we obtain an isomorphism $\R f_!^\bullet \isom \R f_!$.

By the adjoint functor theorem, the functors $f_!^i$ have a right adjoint $f_i^!$. Since $f_!^i$ is exact, $f_i^!$ maps injectives to injectives. By the construction \cite[XVII~1.1.12]{SGA4.3} they form a complex of functors $f_\bullet^!\colon \Ascr-{\rm mod} \to K^b(\Ascr_X-{\rm mod})$. Its right derived functor $\R f^!\colon \D^+(S,\Ascr) \to \D^+(X,\Ascr_X)$ is the right adjoint of $\R f_!$, which can be checked on complexes of injective objects (cf.~\cite[XVIII~3.1.4.12]{SGA4.3}).

\begin{definition}
 The functor $\R f^!$ is called the exceptional inverse image.
\end{definition}

\begin{proposition}[{cf.~\cite[XVIII~Prop.~3.1.7]{SGA4.3}}]
 Let $f\colon X \to S $ be a morphism of finite expansion of qcqs schemes. Let $d$ be a non-negative integer such that $f$ is of relative dimension $\leq d$. Then
 \begin{subenv}
  \item If $L \in \D^+(X,\Ascr_X)$ satisfies $\H^i(L) = 0$ for $i \leq k$ then $\H^i (\R f^! L) = 0$ for $i \leq k-2d$.
  \item For every $L \in \D^+(X,\Ascr_X)$ one has $\diminj \R f^!L \leq \diminj L$.
 \end{subenv}
\end{proposition}
\begin{proof}
 The proof is the same as in the finite type case, as it follows formally from the fact that $f_\bullet^!$ is supported on $[-2d,0]$ and that $f_i^!$ transforms injectives into injectives.
\end{proof}

\begin{remark}
 It is a classical result that $f^! = f^\ast$ when $f$ is \'etale. Note that this no longer holds true for weakly \'etale morphisms of finite expansion, as in general $f_!$ is not left-adjoint to $f^\ast$.
\end{remark}

 The exceptional inverse image satisfies the usual compatibility with the other five operations. The construction of the isomorphisms is again identical to the finite type case. As above, we recall the construction for the reader's convenience; but refer the reader to \cite{SGA4.3} for proofs.

 First, we briefly recall the construction of the adjunction morphism 
 \[
  \R f_\ast \R\Hom(K,\R f^! L) \isom \R\Hom(\R f_! K,L)
 \]
 for $K \in \D^-(X,\Ascr), L \in \D^+(X,\Ascr)$ as presented in \cite[XVIII 3.1.9]{SGA4.3}. First, we represent $L$ by a bounded below complex of injective sheaves (also denoted $L$). Note that
 \begin{equation} \label{eq RHom}
  \R f_\ast \R\Hom(K,  L) \cong \R f_\ast \Hom(K,L) \cong f_\ast \Hom(K,L)
 \end{equation}
 where the latter isomorphism holds as $\Hom(K,L)$ is a complex of flasque sheaves. Let $f = p \circ j$ as above and $\iota\colon V \to S$ be an \'etale morphism. By functoriality of $f_{V,!}$, we obtain a canonical morphism
 \begin{align*}
 f_\ast\underline\Hom^\bullet (K,L)(V) &= \underline\Hom^\bullet(K,L)(X_V) = \Hom^\bullet(\iota_X^\ast K,\iota_X^\ast L) \\ &\to \Hom^\bullet(f_{V\,!}\iota_X^\ast K,f_{V\,!}\iota_X^\ast L) = \underline\Hom^\bullet(f_!^\bullet K, f_!^\bullet L)(V)
 \end{align*}
and thus a morphism of functors $f_\ast\underline\Hom^\bullet(K,L) \to \underline\Hom^\bullet(f_!^\bullet K, f_!^\bullet L)$. By (\ref{eq RHom}) this induces the morphism $\R f_\ast\R \Hom(K,L) \to \R\Hom(\R f_! K, \R f_! L)$, which after replacing $L$ by $\R f^! L$ and composing with the adjunction $\R f_!  \R f^! L \to L$ becomes the wanted morphism
\begin{equation} \label{eq adj1}
 \R f_\ast \R\Hom(K,\R f^! L) \to \R\Hom(\R f_! K,\R f_!  \R f^! L ) \to \R\Hom(\R f_! K,L)
\end{equation}

\begin{proposition}[cf.~{\cite[XVIII Prop.~3.1.10]{SGA4.3}}]
 The morphism (\ref{eq adj1}) is an isomorphism.
\end{proposition}
\begin{proof}
 The proof for the finite type case is purely formal and thus still works in our case. 
\end{proof}

 Next, let $f =  p \circ j$ be a compactification of $f$ and $g\colon S' \to S$ another morphism. We consider the diagram
 \begin{center}
  \begin{tikzcd}
   X' \arrow{r}{g'} \arrow[bend right]{dd}[swap]{f'} \arrow{d}{j'} & X \arrow{d}{j} \\
   \Xbar' \arrow{r}{\gbar} \arrow{d}{p'} & \Xbar \arrow{d}{p} \\
   S' \arrow{r}{g} & S
  \end{tikzcd}
 \end{center}
 with Cartesian squares. Let $\Bscr$ be another torsion sheaf of $S'$ and $M$ a complex of $(\Ascr_{S'},\Bscr)$-bimodules which is bounded above. Let $K$ be any complex of $\Ascr_X$-modules. Following \cite[XVIII 3.1.11]{SGA4.3}, we consider the base change morphism
 \begin{align*}
  M \otimes_{\Ascr_{S'}} g^\ast f_!^\bullet K 
  &= M \otimes_{\Ascr_{S'}} g^\ast p_\ast \tau_{\leq 2d} C_l^\ast j_! K \\
  &\xrightarrow[\rm Prop.~\ref{prop proper BC 2}]{\sim} p'_\ast(p'^\ast M \otimes_{\Ascr_{S'}}  \tau_{\leq 2d} C_l^\ast \gbar^\ast j_! K )\\  
  & \to p'_\ast( \tau_{\leq 2d} C_l^\ast  p'^\ast M \otimes_{\Ascr_{S'}}  \tau_{\leq 2d} C_l^\ast \gbar^\ast j_! K ) \\
  &= p'_\ast \tau_{\leq 2d} C_l^\ast  (p'^\ast M \otimes_{\Ascr_{S'}}  \gbar^\ast j_! K ) \\  
  & \xrightarrow[(\ref{eq BC special})]{\sim} p'_\ast \tau_{\leq 2d} C_l^\ast (j'_! (j'_\ast p'_\ast M \otimes_{\Ascr_{S'}} g'^\ast K)) \\
  & = f_!^\bullet(f'^\ast M \otimes_{\Ascr_{S'}} g'^\ast K), 
 \end{align*}
 as morphisms of functors evaluated at $K$. Thus we obtain a morphism of their respective right adjoints, which evaluated at a complex of $\Bscr$-modules $L$ yields
 \[
  g'_\ast \Hom_\Bscr(f'^\ast,f_{\bullet}'^! L) \to f_\bullet^! g_\ast \Hom_{\Bscr} (M,L).
 \]
 Taking the derived functor at both sides, we obtain a natural morphism for $L \in \D^+(S',\Bscr)$
 \begin{equation} \label{eq adj2}
  \R g'_\ast \R\Hom_\Bscr(f'^\ast M,\R f'^! L) \to \R f^! \R g_\ast \R \Hom_{\Bscr}(M,L).
 \end{equation}
 
 \begin{proposition}[cf.~{\cite[XVIII Prop.~3.1.12]{SGA4.3}}]
  The morphism (\ref{eq adj2}) is an isomorphism
 \end{proposition}
 \begin{proof}
  The proof is the same as in the finite type case, as it follows from a diagram chase using only the base change theorem as input.
 \end{proof}

 We obtain two classical compatibilities of the above isomorphism. If $S = S'$ and $\Bscr = \Ascr$, we obtain
 \[
  \R\Hom(f^\ast K , \R f^! L) \isom \R f^! \R\Hom(K,L).
 \]
 For general $S'$ and  $\Bscr = \Ascr_X = M$, the above isomorphism becomes
 \[
 \R g'_\ast\R f'^! L \isom \R f^! \R g_\ast L.
 \]

\section{Comparison results for the \'etale cohomology} \label{sect-comparison}

\subsection{Comparison with ad-hoc constructions} \label{sect-comparison-ad-hoc}

Let $f\colon X \to \Spec k$ be a separated morphism of finite expansion where $k$ is an separably closed field. We extend the definition of compactly supported cohomology on $X$ given in the previous section to $l$-adic sheaves following Jannsen's construction in \cite{Jannsen:EtCoh}. Denote by $(\ZZ_X\textnormal{-mod})^\NN$ the category of projective systems $(\Fscr_k)_{k \in \NN}$ of Abelian \'etale sheaves on $X$. Note that $(\ZZ_X\textnormal{-mod})^\NN$ contains the category of $l$-adic sheaves. By \cite[\S~1]{Jannsen:EtCoh} the functor $f_\ast \colon (\ZZ_X\textnormal{-mod})^\NN \to (Ab), (\Fscr_k)_{k\in\NN} \mapsto \varprojlim \Gamma(X,\Fscr_k)$ has a right derived functor $\R f_\ast$. We fix a factorisation $f = p \circ j$ where $j$ is an open embedding an $p$ a separated universally closed morphism. For any $l$-adic sheaf $\Fscr = (\Fscr_k)_{k\in \NN}$ on $X$ we define
\[
 \R\Gamma_c (\Fscr) \coloneqq \R p_\ast (j_! \Fscr).
\]
Note that this construction does not depend on the choice of $p$ and $j$ as a consequence of the base change theorem for universally closed morphisms.

\begin{definition}
 Let $l$ be a prime and $\Fscr$ an $l$-adic sheaf on $X$. Then we define
 \[
  \H_c^q(X,\Fscr) \coloneqq \R^q \Gamma_c (\Fscr)
 \] 
\end{definition}

We choose a presentation $X = \varprojlim X_\lambda$ as the limit of finite type schemes with finite transition morphisms $f_{\lambda,\lambda'}$ and  assume that $\Fscr$ is a torsion sheaf or $l$-adic sheaf on the \'etale site of $X$ which is the pullback of a sheaf $\Fscr_{\lambda_0}$ on the \'etale site of $X_{\lambda_0}$.
 In this case one usually uses the ad hoc definition
\[
 H_c^q(X,\Fscr) \coloneqq \varinjlim \H_c^q(X_\lambda, f_{\lambda,0}^\ast \Fscr).
\]
If $\Fscr$ is a torsion sheaf, this definition is compatible with our general definition given in the previous section by the proposition below. If $\Fscr$ is $l$-adic, our definition above yields Emerton's compactly supported completed cohomology groups
 \[
  \Htilde_c^q(X,\Fscr) \coloneqq \varprojlim \varinjlim \H_c^q(X_\lambda, f_{\lambda,0}^\ast \Fscr \otimes \ZZ/l^k\ZZ),
 \]
 see for example \cite{CalegariEmerton:ComplCoh} for further discussions of its properties. 
 
 \begin{proposition} \quad
  \begin{subenv}
   \item If $\Fscr$ is a torsion sheaf, there exists a natural isomorphism $\H_c^q(X,\Fscr) \isom H_c^q(X,\Fscr)$.
   \item If $\Fscr$ is an $l$-adic sheaf, there exists a natural isomorphism $\H_c^q(X,\Fscr) \isom \Htilde_c^q (X,\Fscr)$.
  \end{subenv}
 \end{proposition} 
 \begin{proof}
  First assume that $\Fscr$ is a torsion sheaf. We denote by $f_{\lambda_0}\colon X \to X_{\lambda_0}$ the canonical projection. Then  the canonical morphism $ \varinjlim \R f_{\lambda,\lambda_0\, \ast}f_{\lambda,\lambda_0}^\ast \Fscr_{\lambda_0} \to \R f_{\lambda_0\, \ast} \Fscr $ is an isomorphism by \cite[VII,~Cor.~5.11]{SGA4.2} 
  and hence we get
  \[
\begin{multlined}
 \R\Gamma_c(X,\Fscr) = \R\Gamma_c (X,\R f_{\lambda_0,\ast} \Fscr)  \\ \cong \varinjlim \R\Gamma_c (X_{\lambda_0}, \R f_{\lambda,\lambda_0\, \ast}f_{\lambda,\lambda_0}^\ast \Fscr_{\lambda_0}) = \varinjlim \R\Gamma_c(X_\lambda, f_{\lambda,\lambda_0}^\ast \Fscr_{\lambda_0})).
\end{multlined}
  \]
This yields on cohomology groups the wanted natural isomorphism  $\H_c^q(X,\Fscr) \cong H_c^q(X,\Fscr)$. 

  If $\Fscr = (\Fscr_n)_{n \in \NN}$ is an $l$-adic sheaf, we have
  \begin{align*}
   \H_c^q(X,\Fscr_{n+m}) \otimes \ZZ/l^n \ZZ &\cong \varinjlim ( \H_c^q(X_\lambda,f_{\lambda,\lambda_0}^\ast\Fscr_{\lambda_0,n+m}) \otimes \ZZ/l^n\ZZ) \\ &\cong \varinjlim \H_c^q(X_\lambda,f_{\lambda,\lambda_0}^\ast\Fscr_{\lambda_0,n})  \cong \H_c^q(X,\Fscr_\lambda)
  \end{align*}
  In particular the compactly supported cohomology groups of $\Fscr$ form a Mittag-Leffler system and thus $\H_c^q(X,\Fscr) \cong \varprojlim \H_c^q(X,\Fscr_n)$ by \cite[Prop.~1.6]{Jannsen:EtCoh}. Thus the second assertion follows from the first part of the proposition.
 \end{proof}

\begin{remark}
 We also gave an ad-hoc definition of the functor $\R f_!$ for torsion sheaves in \cite[Appendix B]{HamacherKim:Mantovan}. While it covers all morphisms of finite expansion, we only proved compatibility results for some special cases.
\end{remark}

\subsection{Comparison with sheaf cohomology on the analytification} \label{sect-comparison-Artin}

Let $X$ be a scheme over $\CC$. Then its analytification $X_{\rm an}$ is defined to be the topological space\footnote{Note that, despite the notation, $X_{\rm an}$ is in general not an analytic space unless $X$ is of finite type} with underlying set $X(\CC)$ and the coarsest topology such that for every open subscheme $U \subset X$ and $f \in \Oscr_X(U)$ the subset $U(\CC) \subset X(\CC)$ is open and the induced map $f_{\rm an}\colon U(\CC) \to \AA^1(\CC) = \CC$ is continuous for the analytic topology on $\CC$. In the case that $X$ is affine the topology coincides with the initial topology with respect to $\{f_{\rm an}\colon X(\CC) \to \CC \mid f \in \Oscr_X(X)\}$. Indeed, any principal open set $D(f) \subset X$ is open in the initial topology as $D(f) = f^{-1}(\CC\setminus \{0\})$ and any section $\frac{g}{f} \in \Oscr_X(D(f))$ is continuous as it is the quotient of two continuous functions. One deduces that if $X$ is of finite type, the topology on $X_{\rm an}$ coincides with the usual definition of analytic topology since the latter is defined on an open affine cover as the initial topology with respect to the coordinate functions. 

One easily checks by reducing to the affine case that for every morphism of $\CC$-schemes $\phi\colon X \to Y$ the induced morphism $\phi_{\rm an}\colon X_{\rm an} \to Y_{\rm an}$ is continuous. Thus the analytification defines a functor $(-)_{\rm an}$  from the category of $\CC$-schemes to the category of topological spaces. It satisfies the following exactness properties.
\begin{lemma}
 $(-)_{\rm an}$ commutes with fibre products and limits of filtered projective systems having affine transition maps.
\end{lemma}
\begin{proof}
 In order to show the first assertion, let $f\colon X \to S$, $g \colon Y \to S$ be morphisms of $\CC$-schemes. By the universal property of the fibre products, we have $(X \times_S Y)(\CC) = X(\CC) \times_{S(\CC)} Y(\CC)$ as sets. The canonical bijection $(X \times_S Y)_{\rm an} \to X_{\rm an} \times_{S_{\rm an}} Y_{\rm an}$ is induced by the continuous maps $f_{\rm an}$ and $g_{\rm an}$ and is thus continuous itself. Hence it remains to show that the topology on $X_{\rm an} \times_{S_{\rm an}} Y_{\rm an}$ is finer than the topology on $(X \times_S Y)_{\rm an}$, which may be checked after passing to an affine cover. Thus we assume that $X,Y$ and $S$ are affine. Then we have to show that for any global section $\sum (f_i \otimes g_i) \in \Oscr_{X \times_S Y}(X \times_{S} Y)$ the associated map $X_{\rm an} \times_{S_{\rm an}} Y_{\rm an} \to \CC, (x,y) \mapsto \sum f_i(x) g_i(y)$ is continuous. But this is true as it is an arithmetic expression of continuous functions.
 
 The second assertion is proven analogously to the first. Let $X  = \varprojlim X_\lambda$ be a limit of a projective system of $\CC$-schemes with affine transition maps. By the universal properties of limits, we have that $X(\CC) = \varprojlim X_\lambda(\CC)$ and that the canonical bijection $X_\an \to \varprojlim X_{\lambda,\an}$ is continuous. Thus it remains to show that the topology on $\varprojlim X_{\lambda,\an}$ is finer than the topology on $X_\an$, which may be proven under the assumption that all $X_\lambda$ (and thus $X$) are affine. Since $\Oscr_X(X) = \varinjlim \Oscr_{X_\lambda}(X_\lambda)$ for any $f \in \Oscr_X(X)$ the induced function $X_\an \to \CC$ factorizes as \[ X_\an \xrightarrow{can.} X_{\lambda, \an} \xrightarrow{f_{\lambda,\an}} \CC,\] where $f_\lambda \in \Oscr_{X_\lambda}(X_\lambda)$ and thus is continuous, finishing the proof.
 
 \end{proof}
 
\begin{remark}
 The above construction also holds when one replaces $\CC$ by an arbitrary topological field $k$.
\end{remark}

 We fix a separated scheme $X$ of finite expansion over $\CC$. We denote by $X_\et$ the (small) \'etale site over $X$ and by $\Top(X_\an)$ the site of local isomorphisms over $X_\an$.

 \begin{propdef} \label{propdef analytification}
  The analytification induces a morphism $\nu_X\colon \Top(X_\an) \to X_\et$ of sites, such that for any morphism of $\CC$-schemes $f\colon X \to S$  of finite expansion the diagram
  \begin{center}
   \begin{tikzcd}
    \Top(X_\an) \arrow{r}{\nu_X} \arrow{d}{f_\an} & X_\et \arrow{d}{f} \\
    \Top(S_\an) \arrow{r}{\nu_S} & S_\et
   \end{tikzcd}
  \end{center}
  commutes.
 \end{propdef}
 \begin{proof}
  The main point is to show that the analytification of an \'etale morphism $g\colon U \to X$ is a local isomorphism. Since being \'etale is compatible with limits, $g$ is the base change of an \'etale morphism $g_0\colon X_0 \to Y_0$ of separated $\CC$-schemes of finite type. Since $(-)_\an$ commmutes with fibre products, the $g_\an$ is a local isomorphism as $g_{0,\an}$ is a local isomorphism (see e.g.\ \cite[Exp.~XI, 4.0]{SGA4.3}). 
  
 In other words, analytification is a continuous functor and hence $\nu_{X,\ast}$ and $\nu_X^\ast$ are defined. The commutativity of above diagram (as diagram of continuous functors) is easily checked. It remains to show that $\nu_X^\ast$ is exact. Since exactness can be checked on stalks, it suffices to show that for any $x \in X(\CC)$ and sheaf $\Fscr$ on $X_\et$, we have $(\nu_X^\ast \Fscr)_x = \Fscr_x$. By the commutativity of the above diagram this claim is reduced to the case $X = \Spec \CC$, where it is trivial.
 \end{proof}

 In order to compare compactly supported cohomology, we need the following lemma.
 \begin{lemma} \label{lemma GAGA proper}
  Let $f\colon X \to S$ be a universally closed morphism of separated $\CC$-schemes of finite expansion. Then $f_\an$ is proper.
 \end{lemma}
 \begin{proof}
  This is known to be true if $f$ is proper. Since we can write $f$ as limit of proper morphisms by Temkin's decomposition of universally closed morphisms (Proposiion~\ref{prop temkin decomposition}), the claim follows.
 \end{proof}

  Given a diagram as in Proposition~\ref{propdef analytification} we obtain an exchange morphism of functors on  $K \in \D^+_{\rm tors}(X,\ZZ_X)$
  \begin{equation}
   \nu_S^\ast \R f_\ast \to \R f_{\an, \ast} \nu_X^\ast \label{BC normal}
  \end{equation}
  If $X$ and $S$ are of separated finite type complex schemes and $f$ is proper, Artin's comparison theorem states that (\ref{BC normal}) is an isomorphism. As a consequence, one obtains an isomorphism
  \begin{equation}
   \nu_S^\ast \R f_! \to \R f_{\an, !} \nu_X^\ast. \label{BC shriek}
  \end{equation}
  whenever $f$ is an morphism of finite type schemes over $\CC$. The result can be generalised to schemes of finite expansion.
  
  \begin{theorem}
   Let $f\colon X \to S$ be a universally closed morphism of separated $\CC$-schemes of finite expansion. Then (\ref{BC normal}) is an isomorphism.
  \end{theorem}
  \begin{proof}
   Since $\R f_\ast$ commutes with filtered colimits it suffices to check that (\ref{BC normal}) is an isomorphism for bounded complexes of constructible sheaves $K \in \D^b_{{\rm tors},c}(X,\ZZ_X)$.
  
  We first assume that $f$ is a proper morphism. Then both $f$ and $K$ are defined over finite type complex schemes, i.e. there exists a commutative cube
  \begin{center}
   \begin{tikzcd}[row sep={20}, column sep={20}]
    & X_{\an} \ar{rr}\ar{dd}{g_\an} \ar{dl}[swap]{f_\an} & & X \ar{dd}{g} \ar{dl}{f} \\
    S_{\an} \ar[crossing over]{rr} \ar{dd}[swap]{g'_\an} & & S  \\
    &  X'_{\an}  \ar{rr}\ar{dl}[swap]{f'_\an}  & &  X' \ar{dl}{f'} \\
    S'_\an \ar{rr} && S'  \ar[<-,crossing over]{uu}[swap]{g'},
   \end{tikzcd}
  \end{center}
 where the squares on the left and right side are Cartesian and moreover there exists a $K' \in \D^b_{{\rm tors},c}(X',\ZZ_{X'})$ such that $K = g^\ast K'$. We denote $h \coloneqq g \circ \nu_X$ and $h' \coloneqq g' \circ \nu_S$.  The exchange morphisms give rise to a commutative diagram
 \begin{center}
  \begin{tikzcd}
   \nu_S^\ast g'^\ast \R f'_\ast \arrow{r}{\sim} \arrow[equal]{d}& \nu_S^\ast \R f_\ast g^\ast \arrow{r} & \R f_{\rm an,\ast}  \nu_X^\ast g^\ast \arrow[equal]{d} \\
   h'^\ast \R f'_\ast \arrow[equal]{d} \arrow{rr} & & \R f_{\an,\ast} h^\ast \arrow[equal]{d} \\
   g_{\an}^\ast \nu_S^\ast \R f_\ast' \arrow{r}{\sim} & g_{\an}^\ast \R f_{\an,\ast} \nu_X^\ast \arrow{r}{\sim} & \R f_{\an,\ast} g_{\an}^\ast \nu_X^\ast.
  \end{tikzcd}
 \end{center}
 Here the bottom left morphism is an isomorphism by Artin's comparison theorem for complex varieties, the bottom right morphism is an isomorphism by the proper base change theorem in topology and the top left morphism is an isomorphism by proper base change theorem in algebraic geometry. Thus the top right morphism is also an isomorphism, which we wanted to show.

 Now consider the general case that $f$ is separated universally closed. We write $X = \varprojlim X_\lambda$ as limit of proper $S$-schemes and denote by $f_\lambda\colon X_\lambda \to S$ and $g_\lambda\colon X \to X_\lambda$ the obvious morphisms. We choose $\lambda_0$ such that there exists a $K' \in \D^+_{{\rm tors},c} (X_{\lambda_0},\ZZ_{X_{\lambda_0}})$ such that $K = g_{\lambda_0}^\ast K$. Now the exchange morphisms give rise to a commutative diagram
 \begin{center}
 \begin{tikzcd}
  \nu_S^\ast \R f_{\an,\ast} \arrow[equal]{d} \arrow{rr} & & \nu_X^\ast \R f_{\an,\ast} \nu_X^\ast \arrow[equal]{d} \\
  \nu_S^\ast \R g_{\lambda_0,\ast} \R f_{\lambda_0,\ast} \arrow{r} & \R g_{\lambda_0,\an,\ast} \nu_{X_{\lambda_0}}^\ast \R f_{\lambda_0,\ast} \arrow{r} & \R g_{\lambda_0,\an,\ast} \R f_{\lambda_0,\an,\ast} \nu_X^\ast .
 \end{tikzcd}
 \end{center}
 Here the bottom left morphism is an isomorphism by our considerations above. Thus we assume $S=X_{\lambda_0}$ from now on. We obtain a commutative diagram of functors
 \begin{center} 
  \begin{tikzcd}
   \varinjlim \nu_S^\ast \R f_{\lambda,\ast} f_\lambda^\ast = \nu_S^\ast  \varinjlim \R f_{\lambda,\ast} f_\lambda^\ast \ar{r} \ar{d} & \nu_S^\ast \R f_\ast f^\ast \ar{d} \\
   \varinjlim \R f_{\lambda,\an,\ast} \nu_{X_\lambda}^\ast f_\lambda^\ast = \varinjlim \R f_{\lambda,\an,\ast} f_{\lambda,\an}^\ast \nu_{X_\lambda}^\ast \ar{r} & \R f_{\an, \ast} f_{\an}^\ast \nu_{S,\ast}
  \end{tikzcd}
 \end{center}
 The horizontal morphisms are isomorphisms by \cite[Exp.~VI,~Prop.~8.5.3]{SGA4.2} and the morphism on the left is the exchange morphism (\ref{BC normal}) for $f_\lambda$, which we have shown to be an isomorphism above. Thus the morphism on the right, the exchange morphism for $f$, must also be an isomorphism.
 \end{proof}
  
 \begin{corollary}
  Let $f\colon X \to S$ be a morphism of separated complex schemes of finite expansion. Then we obtain a natural isomorphism $\nu_S^\ast \R f_! \to \R f_{\an, !} \nu_X^\ast$, in particular we have $\H_c^q(X_{\an},\nu_X^\ast \Fscr) \cong \H_c^q(X,\Fscr)$ for every $q \in \NN_0$ and every torsion sheaf $\Fscr$ on the \'etale site of $X$.
 \end{corollary}
 \begin{proof}
  We choose $f = \fbar \circ j$, where $j\colon X \mono \Xbar$ is an open embedding and $\fbar\colon \Xbar \to S$ universally closed. One easily checks that $j_{\an,!} \circ \nu_{X} \cong \nu_{\Xbar} \circ j_!$ and thus
  \[
   \nu_S^\ast \R f_! = \nu_S^\ast \R \fbar_\ast j_! \cong \R \fbar_{\an,\ast} \nu_{\Xbar}^\ast j_! \cong  \R \fbar_{\an,\ast}  j_{\an,!} \nu_{X}^\ast = \R f_{\an, !} \nu_{X}^\ast. 
  \]
 \end{proof}

\def\cprime{$'$}
\providecommand{\bysame}{\leavevmode\hbox to3em{\hrulefill}\thinspace}
\providecommand{\MR}{\relax\ifhmode\unskip\space\fi MR }
\providecommand{\MRhref}[2]{%
  \href{http://www.ams.org/mathscinet-getitem?mr=#1}{#2}
}
\providecommand{\href}[2]{#2}

\end{document}